%% file: ECC18.tex
\documentclass[letterpaper, 10 pt, conference]{ieeeconf} 
\IEEEoverridecommandlockouts
\input{preamble}

\title{\LARGE \bf
Consensus over evolutionary graphs
}
\author{Michalis Smyrnakis, Nikolaos M. Freris and Hamidou Tembine%
\thanks{The authors are with the Division of Engineering at New York University Abu Dhabi, Saadiyat Island, P.O. Box 129188, Abu Dhabi, UAE. The last two authors are also with NYU Tandon School of Engineering. 
        E-mails: \texttt{\{ms10775, nf47, tembine\}@nyu.edu}. %
        This work was supported by the National Science Foundation (NSF) under grant CCF-1717207, and the U.S. Air Force Office of Scientific Research (AFOSR) under grant FA9550-17-1-0259.
        }
}

\begin{document}
\maketitle
\begin{abstract}
We establish average consensus on graphs with dynamic topologies prescribed by evolutionary games among strategic agents. Each agent possesses a private reward function and dynamically decides whether to create new links and/or whether to delete existing ones in a selfish and decentralized fashion, as indicated by a certain randomized mechanism. This model incurs a time-varying and state-dependent graph topology for which traditional consensus analysis is not applicable. We prove asymptotic average consensus almost surely and in mean square 
for any initial condition and graph topology. In addition, we establish exponential convergence in expectation. Our results are validated via simulation studies on random networks.

\end{abstract}


\section{Introduction}
\emph{Evolutionary game theory} has been established as a modeling tool for interactions between populations of strategic entities. In specific, evolutionary games describe the population dynamics resulting from pairwise interactions. It has found numerous applications in various areas of multi-agent systems such as in wireless networks~\cite{tembine1,tembine2}, swarm robotics~\cite{sun2015}, and  dynamic routing protocols~\cite{tembine3}.

\emph{Evolutionary graphs} arise as an application of evolutionary game theory in modeling dynamic graph topologies. In such context, a population is organized as a network (graph) with the nodes (vertices) representing atoms (agents) and links (edges) representing interactions among them. This is captured by a weighted 
graph with time-varying edge set determined by an evolutionary game. We will present a specific randomized decentralized mechanism for determining the graph topology based on the individual fitness functions of the agents. In our setting, each node maintains a local variable and computes its fitness function using its own value along with the values from its neighbors (both active and inactive, cf. Sec.~\ref{sec:problem}). Subsequently, it dynamically readjusts its neighbor set by randomly adding or deleting links with probabilities dictated by the resulting change of its fitness function.

\emph{Consensus} is a canonical example of in-network coordination in multi-agent systems. Each agent maintains a local value and the goal is for the entire network to reach agreement to a common value in a \emph{distributed} fashion, i.e., via local exchanges of messages between neighboring (adjacent) agents. An archetypal problem is \emph{average consensus}, where the goal is for each agent to asymptotically compute the average of all nodal values;  cf.~\cite{linear_cons,xiao2004fast,boyd2006randomized,olfati},  for a largely non-exhaustive list of references.    
This theme has proven a prevalent design tool for distributed multi-agent optimization~\cite{nedic2009distributed,nedich2016achieving}, signal processing~\cite{gossip},  numerical analysis~\cite{REK}, and estimation~\cite{RK,RK2}. 

In this paper, we seek to bridge the gap between evolutionary game theory and distributed optimization by studying average consensus over evolutionary graphs. In our setting, each agent has a local variable that represents its `strategy': the strategies evolve following a consensus protocol over a time-varying graph capturing inter-agent cooperations. At each time instant, agents dynamically select the agents with which they cooperate (from a candidate set) based on their utility (i.e., fitness) that depends on their own strategy and the strategies of neighboring agents they intend to cooperate with. Specifically, agents create or drop links (i.e., cooperations) when they deem it beneficial for them, and they do so via a randomized decentralized mechanism, in a selfish manner. Examples enlist social networks~\cite{sn1} and coordination of robot swarms~\cite{sun2015}. 

We consider the problem of average consensus over networks  
of time-varying topology  
captured by a Markov Decision Process (MDP).
Unlike prior work on the subject~\cite{olfati,nedich2016achieving}, the topology \emph{depends on the agents' values}, which renders previous analysis techniques inapplicable in our case. We proceed to establish average consensus a.s. (almost surely) and in m.s. (mean square) sense, using  stochastic Lyapunov techniques. Additionally, we prove that the convergence is exponential in expectation, and provide a lower bound on the expected convergence rate. Finally, our method was empirically assessed via numerical simulations.        

The remainder of the paper is organized as follows: Sec.~\ref{sec:preliminaries} exposes preliminaries on graph theory, consensus and evolutionary graph theory. In Sec.~\ref{sec:problem}, we present the problem formulation. The convergence analysis is presented in Sec.~\ref{sec:proof}.  Sec.~\ref{sec:sim} illustrates simulation results, while Sec.~\ref{sec:conclusions} concludes the paper and discusses future research directions.

\section{Preliminaries}\label{sec:preliminaries}
In this section, we recap essential background on graph theory, consensus protocols and evolutionary games. In the remainder of the paper, we use boldface for vectors and capital letters for matrices. Vectors are meant as column vectors, and  we use $\zero,\one$ to denote the vectors with all entries equal to zero, one, respectively, and $I$ to denote the identity matrix (with the dimension made clear from the context in all cases). 
Last, we use the terms `agent', `vertex' and `node' interchangeably, and the same holds true for `edge' and `link,' in what follows.  

\subsection{Graph theory}

Consider the case of $n$ interacting agents which aim to achieve consensus over a quantity of interest, for instance compute the average of their values. 
Such problem is an instance of \emph{computing on graphs}, where each agent is modeled as a vertex of the graph and edges are drawn only between interacting nodes: we assume that two nodes can interact with each other (for instance, exchange private information) if and only if they are connected, i.e., there is an edge between them in the graph. 

Formally, let a graph be denoted by $\G=(\mathcal{V}, \mathcal{E})$, where $\mathcal{V}$ is the non-empty set of vertices (nodes) and $\mathcal{E}$ is the set of edges. 
In this paper, we restrict attention to \emph{undirected} graphs, that is to say the edge set $\Ed$ consists of unordered pairs: $(i,j)\in \Ed$ implies that agents $i,j$ can interact in a symmetric fashion, i.e., cooperate\footnote{The extension of our methods to directed graphs that model asymmetric interactions (e.g.,  asymmetric reward functions) will be the focal point of future work.}. We further assume that the graph does not contain self-loops (i.e., $(i,i)\notin \mathcal{E}$ for all $i\in \V$); this is without any loss in generality, since edges capture inter-agent interactions in our framework.     
We set $n:=|\V|, m=|\Ed|$, to denote the number of nodes, edges respectively. We say that the graph $\G$ is \emph{connected} if there is a path between any two nodes $i,j\in\V$; otherwise, we say that the graph is disconnected.

The adjacency matrix $\A\in\R^{n\times n}$ of the graph captures connections between the nodes: 
for any two nodes $i,j \in \mathcal{V}$, $a_{ij}$ is defined as:
$$a_{ij}=\left\lbrace \begin{array}{c l}
1, & \textrm{if } (i,j)\in\Ed,\\
0, & \textrm{otherwise.}
\end{array} \right.$$
The definition can be extended to \emph{weighted} graphs, in which case $a_{ij}$ can take an arbitrary 
value if $(i,j)\in\Ed$. For an undirected graph, $a_{ij}=a_{ji}$, for all $i,j\in \V$, i.e., $\A$ is symmetric. Besides, $\A$ has a zero-diagonal ($a_{ii}=0$ for all $i$), since $\G$ is assumed to have no self-loops. The \emph{neighborhood} of a node $i$  contains all nodes that the node has a connection with, and is denoted by $\N_i := \{j:a_{ij}\ne 0\}$. 
The \emph{degree} of node $i$ is defined as the 
number of its neighbors, i.e.,  $d_i:= \abs{\N_i} = \sum_{j\in \V}a_{ij}$. Let $\D$ be the diagonal degree-matrix (i.e., its $(i,i)-$th entry equals $d_i$, and off-diagonal entries are zero). We define 
$\L := \D-\A,$
the \emph{Laplacian} of the graph. Clearly, $\L\in\R^{n\times n}$ is symmetric.  
Additionally, it can be shown that $\L$ is positive semidefinite~\cite{graph_theory}. 
It is well-known~\cite{graph_theory} that $\rank \L = n-k$, where $k$ is the number of connected components of the graph. In particular, a graph is connected if and only if $\rank \L = n-1.$
By its very definition, the Laplacian has a zero eigenvalue with corresponding eigenvector the all-one vector $\one$. In fact, the multiplicity of the zero eigenvalue equals the number of connected components of the graph. 
The second smallest eigenvalue of the Laplacian is called the \emph{Fiedler value} or \emph{algebraic connectivity} of the graph, and is denoted as $\lambda_2(\G):=\lambda_2(\L)$. It is positive if and only if the graph is connected.

\subsection{Consensus}
Each agent $i\in \V$ maintains a local scalar value $x_i$. We call the \emph{state} of the network the vector $\x$ obtained by stacking all nodal values $\{x_i\}_{i\in\V}$. We use $x_{i,t}$ to denote the value of node $i$ at time $t$, and, correspondingly, $\x_t$ for the network state at time $t$. 

A widely studied method which describes the evolution of $x_{i}$ is \emph{linear consensus}~\cite{linear_cons,olfati}.  
The dynamics for $x_{i}$ can be written (up to a multiplicative constant) as
$$\dot{x}_{i,t} = \sum_{j\in \mathcal{N}_{i}} a_{i,j}(x_{j,t}-x_{i,t}),$$
which can further be written compactly in matrix form:
$$\dot{\x}_t = -\L\x_t.$$
For an undirected and connected graph, the network reaches \emph{average consensus}, i.e.,
$$\x_t \underset{t\to\infty}{\longrightarrow} Ave(\x_0)\one,$$
where $Ave(\x_0):=\frac{1}{n}\one^\top\x_0$ is the average of the values at the initial time 0~\cite{olfati}. Besides, the convergence rate is exponential with rate lower-bounded by the Fiedler value~\cite{olfati}. 

In this paper, we will study consensus over \emph{time-varying} graphs as abstracted by a time-varying Laplacian $\L_t$, dictated by agents' randomized decisions; that is to say, $\L_t = \L(\x_t,\Ed_{t^-})$ is a random matrix that depends on both the state at time $t$ and the topology `right before' time $t$;  consequently, the analysis in existing literature~\cite{linear_cons,xiao2004fast,boyd2006randomized,olfati,nedich2016achieving} does not directly carry through.

\subsection{Evolutionary graphs}
An \emph{evolutionary graph} is a graph whose topology is specified by an evolutionary game~\cite {egraph1,egraph2} of a single population with finite number of players, $n$, placed on a graph $\mathcal{G}$. The interactions of players are captured by the edge set of $\mathcal{G}$. Each player $i$ has a set of actions $s_{i} \in S_{i}$ and receives a certain pay-off according to its utility (or fitness) function which is a mapping from the joint action space 
$S:= S_1\times S_2\times \hdots S_n$ 
to the real numbers, $f_{i}:S\rightarrow \mathbb{R}$. 

The evolution of graph topology follows a stochastic process administered through pairwise interactions. In particular, two players (that are allowed to interact) are randomly selected and a ``copy'' of the player with the higher fitness, takes the place of the player with the lower one. As a result, the strategy of the player with the smaller fitness is replaced by the strategy of the player with the higher one. 

\section{Problem formulation}\label{sec:problem}

In a cyberphysical system, such as a wireless sensor network~\cite{mccabe2008controlled,freris2010fundamentals}, it is common that agents may opt to dynamically create new links with other agents or drop existing ones during the coordination process; this behavior results in time-varying graphs. In order to properly describe this process, it is necessary to formulate a dynamic graph, whose structure depends on time, topology and network state. 

In what follows, we define a \emph{dynamic graph} as a graph with fixed predetermined vertex set $\V$ of agents, 
in which the edge set $\Ed$ varies over time, in a \emph{state-dependent randomized} fashion.  In particular, the edge set $\Ed = \Ed(\x_t,t,\omega) \subseteq \V\times\V$ is a state-/time-dependent random set in a probability space $(\Omega,\F,\P)$ (where $\Omega$ is the sample space, $\F$ is the $\sigma-$algebra on $\Omega$, and $\P$ is the probability measure); correspondingly, we define the adjacency matrix $\A= \A(\x_t,t,\omega)$ and the Laplacian matrix $\L = \L(\x_t,t,\omega)$.  In this paper, we focus on time-varying graphs with topology  at time $t$ depending on both the state at time $t$ and the topology `right before' time $t$, i.e.,  
$\L = \L(\x_t,\Ed_{t^-},\omega)$. We use the shorthand notation $\Ed_t,\A_t,\L_t$ and $\G_t = (\V,\Ed_t)$ to emphasize the type-varying aspect. 

For each node $i\in \V$, we denote by $\N_i$ the set of all \emph{feasible neighbors}, i.e., the set of all nodes that node $i$ can potentially create a connection with. Since we focus on undirected graphs, we assume that $ j\in \N_i$ implies that 
$i \in \N_j$. 
 At each time $t$, for any given node $i\in \V$, we denote the set of \emph{active neighbors}, i.e., the set of nodes with which $i$ is connected, using 
$\N_{i,t}^{(1)} \subseteq \N_i$.  Furthermore, we let $\N_{i,t}^{(2)}:=\N_i \setminus  \N_{i,t}^{(1)}$, the set of \emph{inactive neighbors}, i.e., the set of nodes that $i$ is not connected with, but may decide to connect with based on the evolutionary game. 
The degree of a node $i$ at time $t$,  $\abs{\mathcal{N}_{i,t}^{(1)}}$, is the total number of active neighbors of node $i$ at time $t$. For the subsequent treatment, we make the assumption that the graph obtained by taking the union of all feasible neighbor sets is connected:
\begin{Assumption}\label{ass:connectivity}
	The graph $\G' = (\V, \Ed')$, with $\Ed':=\cup_{i\in\V}\cup_{j\in\N_i} \{(i,j)\}$ is connected. 
\end{Assumption}
This condition is necessary for consensus to be achieved: otherwise, the graph will be disconnected at each time regardless of the agents' decisions, with no possible exchange of information across its connected components, which implies that consensus is infeasible. We will establish the sufficiency of the condition for the dynamic topology instructed by the evolutionary game we propose. 

Let $\vecx{x}$ denote the \emph{coordination levels} of the agents, where we will restrict attention (without any loss in generality) to the case that $0 \leq x_{i,t} \leq 1, \forall i \in \mathcal{V},t\ge 0$. For instance, the evolution of the coordination levels may be considered as a resource allocation process, in which agents decide to share a percentage of a resource they own with their neighbors. Similarly, in an opinion dynamics setup, the coordination levels may reflect the beliefs of the agents, e.g.,  the information state of vehicles in a robot team. 

The agents adjust their coordination levels based on interactions with their neighbors, according also to their tendency to create a link with other agents or drop an existing one. The evolution of an agent's coordination level may be described by the following dynamic consensus protocol: 
\begin{equation}
\label{eq:wei_cons_pro}
\begin{array}{r l}
\dot{x}_{i,t}=&
 \sum_{j \in \mathcal{N}_{i,t^-}^{(1)}}\chi_{ij,t}^{m}(x_{j,t}-x_{i,t})  \\ &+  \sum_{k \in \mathcal{N}_{i,t^-}^{(2)}}\chi_{ik,t}^{c}(x_{k,t}-x_{i,t}),
\end{array}
\end{equation}
where $\chi_{ij,t}^{m},\chi_{ik,t}^{c}$ are $0-1$ variables that respectively indicate whether to: a) \emph{maintain} an existing link (i.e., a link that is active `right before' time $t$, equivalently $(i,j)$ with $j\in \mathcal{N}_{i,t^-}^{(1)}$), if $\chi_{ij,t}^{m}=1$ ($\chi_{ij,t}^{m}=0$ means that the link is dropped); and b) \emph{create} a new link $(i,k)$ with $k\in \mathcal{N}_{i,t^-}^{(2)}$, if $\chi_{ik,t}^{c}=1$. 
Clearly, we set $\chi_{ji,t}^{m} \equiv \chi_{ij,t}^{m}, \chi_{ki,t}^{c} \equiv \chi_{ik,t}^{c}$. 
The decisions are Bernoulli random variables with respective `success' probabilities (the probability of the value $1$) given by $0\leq w_{ij,t}^{m}\leq 1$, and $0 \leq w_{ik,t}^{c}\leq 1$. 

In this paper, the decision rules are \emph{state-dependent} and time-invariant, i.e., $\chi_{ij,t}^{m},\chi_{ik,t}^{c}$ are independent Bernoulli random variables with success probabilities that depend on the coordination levels of the two neighbors: 
$w_{ij,t}^{m} = w_{ij,t}^{m} (x_{i,t},x_{j,t})$, $w_{ik,t}^{c} = w_{ik,t}^{c} (x_{i,t},x_{k,t})$, cf.~\eqref{eq:weidrop},~\eqref{eq:weicreate} for their exact definition. 

The following remark underlines the inapplicability of previous analysis~\cite{olfati} in our setting.
\begin{Remark}
	The graph corresponding to the Laplacian matrix $\L_t$ may be disconnected. 
\end{Remark}

Indeed, since decisions are probabilistic, there is a positive probability that the resulting graph is disconnected (even the event of an empty edge set has positive probability) at each given time instant. 

We may stack the decisions $\{\chi_{ij,t}^{m},\chi_{ij,t}^{c}\}$ into a corresponding Laplacian matrix $\L = \L(\x_t,\Ed_{t^-},\omega)\equiv \L_t$ with entries $\{l_{ij}\}_{i,j\in \V}$ (dropping time dependency for notational simplicity) defined by: 
\begin{eqnarray*}
l_{ij}=l_{ji} := &- \left(1_{\{j \in \N_i^{(1)}\}}\chi_{ij}^{m} + 1_{\{j \in \N_i^{(2)}\}}\chi_{ij}^{c}\right),\\
l_{ii} := &-\sum_{j\ne i} l_{ij},
\end{eqnarray*}
where $1_{\{\cdot\}}$ is the $0-1$ indicator function (1 if the event holds and 0 else). We proceed to write the update rule in matrix form as follows: 
\begin{equation}\label{eq:state_evolution}
\dot{\x}_t = -\L_t\x_t.
\end{equation}
We call this the \emph{state evolution} equation; note that, by its very definition, it constitutes a continuous Markov Decision Process (MDP).  

The following proposition shows that all coordination levels are guaranteed to remain in the interval $[0,1]$ if they are initialized in $[0,1]$, i.e., it establishes that the set $[0,1]^n$ is \emph{forward invariant}.
\begin{Proposition}\label{prop:invariance}
	Suppose $\x_0 \in [0,1]^n$. Then, under state evolution~\eqref{eq:state_evolution},  $\x_t \in [0,1]^n$ for all $t>0$, i.e., $[0,1]^n$ is forward-invariant. 
\end{Proposition}

\begin{proof}
The proof considers two cases: the first case considers a nodal value reaching the upper bound (1), and the second the lower bound (0). 

Case 1: Assume that for some $t\ge 0$, there exists $i\in \V$ with $x_{i,t}=1$ and that $\x_s\in [0,1]^n$ for all $s\le t$. Then, given that $x_{j,t} \in [0,1]$ for all $j\ne i$ it follows that
$$\dot{x}_{i} = \sum_{j\in\V\setminus\{i\}} -l_{ij} (x_{j,t}-1) \le 0,$$
because $x_{j,t} \le 1$ and $-l_{ij}\ge 0$. Therefore $x_i$ can never exceed the value 1.

Case 2: Assume that for some $t\ge 0$, there exists $i\in \V$ with $x_{i,t}=0$ and that $\x_s\in [0,1]^n$ for all $s\le t$. Since $x_{j,t} \in [0,1]$ for all $j\ne i$ it follows that
$$\dot{x}_{i} = \sum_{j\in\V\setminus\{i\}} -l_{ij} x_{j,t} \ge 0,$$
therefore $x_i$ can never go below 0.
\end{proof}
\begin{Remark}
	The forthcoming analysis applies irrespectively of the assumption that the values $\x_t \in [0,1]^n$, i.e., for arbitrary initial conditions $\x_0$. This assumption is adopted solely for the sake of interpretability in the context of evolutionary games.
\end{Remark}

\subsection{Evolutionary game}
In this section, we provide a rule for selecting the weights (i.e., probabilities) $w_{ij,t}^{m}, w_{ij,t}^{c}$ based on a particular evolutionary game. We use \emph{Continuous Actions Iterative Prisoner's Dilemma} (CAIPD)~\cite{caipd} to define the fitness function of a given node and illustrate how the weights are selected. 

In CAIPD, there are $n$ agents that choose their coordination levels given their neighbors' decisions. Each agent $i$ has to pay a fee that is related to its coordination level and gains a reward related to the coordination levels of its neighbors: the higher the coordination level of agent $i$ and the coordination levels of its neighbors are, the higher the cost and reward are,  respectively.

Formally, the reward of agent $i$ when it sets its coordination level to $x_{i}$ is defined using the following fitness function (where we drop dependency from time $t$ henceforth, since the definition of fitness in CAIPD is time-independent): 

\begin{equation}
f_{i}(\vec{x})=b\sum_{j \in \mathcal{N}_{i}^{(1)}}{x_{j}}-c\abs{\mathcal{N}_{i}^{(1)}}x_{i},
\label{eq:fitness}
\end{equation}
where $b>c>0$ are constants (i.e., we assume that the gain--per unit of coordination--from cooperating with another agent $b$ is higher than the per-unit loss $c$).

The change in the fitness function of agent $i$ when it creates or drops a link, denoted by $\tilde{f}^{c}_{ij}$ (for $j\in \N_{i}^{(2)}$) and $\tilde{f}^{d}_{ij}$ (for $j\in \N_{i}^{(1)}$) respectively, is determined by:
\begin{eqnarray}
\label{eq:fcreate}
\tilde{f}^{c}_{ij}(\vec{x})&= & b\sum_{k \in \mathcal{N}^{(1)}_{i}\cup\{j\}} x_{k}-c\left(\abs{\mathcal{N}^{(1)}_{i}} +1 \right) x_{i}\nonumber\\
&-&\left(b\sum_{k \in \mathcal{N}_{i}^{(1)}}{x_{k}}-c\abs{\mathcal{N}_{i}^{(1)}}x_{i}\right)\nonumber\\
&=&bx_{j}-cx_{i},
\end{eqnarray}
%
\begin{eqnarray}
\label{eq:fdrop}
\tilde{f}^{d}_{ij}(\vec{x}) &= & b\sum_{k \in \mathcal{N}^{(1)}_{i} \setminus \{j\}} x_{k}-c\left(\abs{\mathcal{N}^{(1)}_{i}} -1 \right) x_{i} \nonumber\\
&-& \left(b\sum_{k \in \mathcal{N}_{i}^{(1)}}{x_{k}}-c\abs{\mathcal{N}_{i}^{(1)}}x_{i}\right) \nonumber\\
&=& cx_{i}-bx_{j}.
\end{eqnarray}

In the evolutionary game, a link may be created/dropped if both agents desire to coordinate or not based on the corresponding  increment (or decrement) of their individual fitness functions. Essentially, if both agents benefit from maintaining/creating a link, the corresponding probability must be higher than the case where only one node benefits or when the fitness of both agents is decreased. In \cite{ecc} a sigmoid function was used to determine the weights  $w_{ij,t}^{m}$ and $w_{ij,t}^{c}$.  In our formulation, the weights correspond to the probabilities that agent $i$ maintains (one minus the probability that it drops) or creates link $(i,j)$, respectively. 
Following a similar approach, the weights 
are selected as (where we once again drop time dependency since the weight-rule is state-dependent but time-invariant):

\begin{equation}
\begin{array}{rl}
\label{eq:weidrop}
w_{ij}^{m}=&\frac{1}{2}-\frac{1}{2}\tanh\left( \tilde{f}_{ij}^{d}(\vec{x})+\tilde{f}_{ji}^{d}(\vec{x})\right )\\
=& \frac{1}{2}-\frac{1}{2}\tanh\left( (c-b)(x_{i}+x_{j})\right ),
\end{array}
\end{equation}
\begin{equation}
\begin{array}{rl}
\label{eq:weicreate}
w_{ij}^{c}=&\frac{1}{2}+\frac{1}{2}\tanh\left( \tilde{f}_{ij}^{c}(\vec{x})+\tilde{f}_{ji}^{c}(\vec{x})\right )\\
=& \frac{1}{2}+\frac{1}{2}\tanh\left((b-c)(x_{i}+x_{j})\right ).
\end{array}
\end{equation}

Note that, by definition, the two values are equal and lower-bounded by $\frac{1}{2}$ (in light of the fact that $\x_t\in[0,1]^n$, for all $t\ge 0$; cf. Proposition~\ref{prop:invariance}).
We define the \emph{weighted Laplacian} matrix $\W$ with entries (again dropping time dependency for notational simplicity) given by:
\begin{eqnarray}\label{eq:w_define}
	w_{ij}=w_{ji} := &- \left(1_{\{j \in \N_i^{(1)}\}}w_{ij}^{m} + 1_{\{j \in \N_i^{(2)}\}}w_{ij}^{c}\right),\\
	w_{ii} := &-\sum_{j\ne i} w_{ij},\nonumber
\end{eqnarray}

\section{Convergence analysis}\label{sec:proof}
Formally, for $t>0$, $\L_t$ is an $\F_{t^-}$-measurable random matrix where the $\sigma-$algebra $\F_{t^-}$ is defined by $\F_{t^-} := \sigma(\cup_{s<t} \sigma(\Ed_s),\x_t) = \sigma( \cup_{s<t}\sigma (\L_s),\x_t)$; $\sigma(\cdot)$ denotes the completion (i.e., adding all subsets of sets of zero measure~\cite{stroock}) of the $\sigma-$algebra generated by its argument (a random variable). Simply said, $\F_{t^-}$ is a formal way of describing the available information pertaining to the topology `right before' time $t$, along with the value $\x_t$  
and it induces a \emph{filtration}~\cite{stroock}, i.e., $\F_{s^-} \subseteq \F_{t^-}$ for $s\le t$. 
In the sequel, we use the notation $\E_t[\cdot]:=\E[\cdot|\F_{t^-}]$;
we assume the initial state $\x_0$ and topology $\Ed_0$ are deterministic and known. 

It follows that
\begin{equation}\label{eq:cond_exp}
\W_t = \E_t[\L_t].
\end{equation}

Note that~\eqref{eq:state_evolution} is a `stochastic differential equation\footnote{We use this term in brackets as the differential equation is driven by the random Laplacian matrix and not a Brownian motion~\cite{stroock}.}' which is equivalent to the integral equation:
\begin{equation}\label{eq:integral_equation}
\x_t = \x_0  -\int_{0}^{t}\L_s\x_s ds.
\end{equation}

The following lemma characterizes the evolution of the 
mean: 
$$\bar{\x}_t:=\E[\x_t].$$
\begin{Lemma}[Evolution of the mean]\label{lemma:mean}
	Under state evolution~\eqref{eq:state_evolution}, the mean value follows the differential equation:
	\begin{equation}\label{eq:mean_evolution}
	\dot{\bar{\x}}_t = -\E[\W_t \x_t].
	\end{equation}
\end{Lemma}
\begin{proof}
Taking expectation in~\eqref{eq:integral_equation}	yields
\begin{eqnarray*}
	\bar{\x}_t & =& \bar{\x}_0  -\E[\int_{0}^{t}\L_s\x_s ds]\\
	& = &\bar{\x}_0  - \int_{0}^{t}\E[\L_s\x_s] ds\\
	& = & \bar{\x}_0  - \int_{0}^{t}\E[\E_s[\L_s\x_s]] ds\\
	& = & \bar{\x}_0  - \int_{0}^{t}\E[\W_s\x_s] ds
\end{eqnarray*}

The first equality uses the definition of $\bar{x}_t$. The second one invokes Fubini's theorem~\cite{royden} (since $\x_t\in [0,1]^n$ and $\L$ is a finite dimensional matrix with bounded entries). The third equality uses the towering property of expectation~\cite{stroock}.  
The fourth uses the fact that $\x_s$ is  $\F_{s^-}-$measurable  
along with~\eqref{eq:cond_exp}.
\end{proof}

The next theorem establishes the convergence of our scheme:
\begin{Theorem}[Average consensus]\label{thm:convergence}
	Under Assumption~\ref{ass:connectivity} and state evolution~\eqref{eq:state_evolution} the system reaches average consensus:
	$$\lim_{t\to \infty} \x_t = Ave(\x_0) \one \ \text{a.s. and in m.s.},$$
	for any $\x_0\in[0,1]^n$, where $Ave(\x_0):=\frac{1}{n}\one^\top\x_0$ is the average of the initial nodal values.
	\hide{
	$$\lim_{t\to \infty} \x_t = Ave(\x_0) \one \ \text{a.s. and in } \ell_p(\Omega,\F,\P),$$
	for any $\x_0\in[0,1]^n$ and for any $p\in [1,\infty]$, where $Ave(\x_0):=\frac{1}{n}\one^\top\x_0$ is the average of the initial nodal values, and for a random vector $\z\in \R^n$
	$$\|\z\|_{\ell_p(\Omega,\F,\P)} := \E[\|\z\|_p].$$}
	Furthermore, the m.s. convergence is exponential in expectation,
	with the expected rate lower-bounded by $\lambda_2(\G')>0$. 
\end{Theorem}

\begin{proof}
	Since $\L$ is symmetric and $\L\one =\zero$, pre-multiplying~\eqref{eq:integral_equation} by $\one^\top$ gives $\one^\top\x_t = \one^\top\x_0$ for all $t\ge 0$, i.e., the sum (and therefore the average) of entries is constant over time. We define the \emph{disagreement} vector $\e_t :=\x_t- Ave(\x_0)\one$: it follows that $\e_t\perp\one$, i.e., $\one^\top\e_t = \zero$ for all $t\ge 0$. Consider the Lyapunov function $V(\e) = \frac{1}{2}\|\e\|_2^2 = \frac{1}{2}\e^\top\e$. Under~\eqref{eq:state_evolution} it follows that:
	$$\dot{V}(\e_t) = -\e_t^\top\L_t\e_t,$$
	where we have used the chain rule and the property that $\L_t\one=\zero$. Note that the drift satisfies $-\e_t^\top\L_t\e_t\le 0$ since $\L_t$ is positive semidefinite. 
	Using the exact same line of analysis as in Lemma~\ref{lemma:mean} we get: 
	$$\E[V(\e_t)] = V(\e_0) - \int_{0}^{t}\E[\e_s^\top\W_s\e_s]ds,$$
	or more generally:
	$$\E_s[V(\e_t)] = V(\e_s) - \int_{s}^{t}\E_s[\e_\tau^\top\W_\tau\e_\tau]d\tau,$$
	Therefore $V(\e_t)$ is a bounded (cf. Proposition~\ref{prop:invariance})  $(\Omega,\mathcal{F}_{t^-},\P)-$supermartingale, and therefore converges a.s. by the supermartingale convergence theorem~\cite{stroock}. Denote the limit by $\e_\infty(\omega)$; we will establish that $\e_\infty=\zero$ a.s. Note that $\W(\x)$ is a (state-dependent) weighted Laplacian on the graph $\G' = (\V, \Ed')$ which is connected (cf. Assumption~\ref{ass:connectivity}), therefore $\lambda_2(\G')>0$. Furthermore, the edge weights are positive and bounded away from zero uniformly over $\x$; to see this note that~\eqref{eq:weidrop},~\eqref{eq:weicreate} and the fact that $\x\in [0,1]^n$ imply that 
	$$\min (w_{ij}^m,w_{ij}^c) \ge \frac{1}{2}.$$
	This also shows that $\lambda_2(\W) \ge \frac{1}{2}\lambda_2(\G')>0$. Since $\e_t\perp \one$ for all $t\ge 0$, and by the definition of $V(\cdot)$, we have: 
	$$\E[V(\e_t)] \le V(\e_0) - \lambda_2(\G')\int_{0}^{t}\E[V(\e_s)]ds,$$ 
	consequently 
	$$\E[V(\e_t)] \le V(\e_0)e^{-\lambda_2(\G') t},$$
	i.e., 
	$$\lim_{t\to \infty}\E[V(\e_t)]=0.$$
	
	This establishes m.s. convergence to $\zero$, with expected exponential convergence with rate lower-bounded by $\lambda_2(\G')$; 
	a.s.-convergence follows by the supermartingale convergence theorem and Fatou's lemma~\cite{royden}.
\end{proof}
\begin{Corollary}[Convergence in expectation]
	Under Assumption~\ref{ass:connectivity} and state evolution~\eqref{eq:state_evolution}:
	$$\lim_{t\to \infty} \E[\x_t] = Ave(\x_0) \one.$$
\end{Corollary}
\begin{proof}
	By Jensen's inequality, $\|E[\e_t]\|_2^2 \le \E[\|\e_t\|_2^2]$, therefore 
	$$\lim_{t\to \infty} \E[\e_t] = \zero,$$
	and the result follows by the definition of $\e_t$.
\end{proof}
\section{Experiments}\label{sec:sim}
In this section, we present simulation studies that attest our convergence results.   
We have employed the \emph{small world network} model~\cite{watts1999small}, i.e., a Bernoulli random graph in which any two agents are allowed to interact with a fixed probability $p_1$; this process generates the neighborhood sets $\{\N_i\}$ and therefore the graph $\G' = (\V, \Ed')$, cf. Assumption~\ref{ass:connectivity}. We took the network size $n=1000$ in our experiments and set $p_1=0.2$; we repeated the experiment until a connected graph $\G'$ was obtained as required by Assumption~\ref{ass:connectivity}. For initialization, we chose $\x_0$ uniformly distributed on $[0,1]^n$ and selected the active neighbor sets $\N_{i}^{(1)}$ as follows: for each $i$, neighbors in $\N_i$ were selected to be active with probability $p_2$ (independently from one another); we took $p_2=0.2$. Last, we set $b=5,c=4$ in~\eqref{eq:weidrop},~\eqref{eq:weicreate}.

For numerical simulation of the state evolution we have performed uniform discretization of~\eqref{eq:state_evolution} with a step-size $\Delta$, i.e., we set $t=\Delta k$ where $k$ is the discrete iterate counter and run:
$$\x_{k+1} = \x_k -\Delta \L_k \x_k.$$
We chose the step-size $\Delta=\frac{1}{n}$ which guarantees that the spectral radius of $(I - \Delta \L_k)$ is less than or equal to 1 (since the eigenvalues of the Laplacian are upper bounded by $n$
~\cite{graph_theory}).  

Figure \ref{fig:fig1} depicts the evolution of the coordination levels (for a single small world network and initialization of $\x_0$): it is evident that all coordination levels converge to the average value. 
\begin{figure}
\includegraphics[scale=0.55]{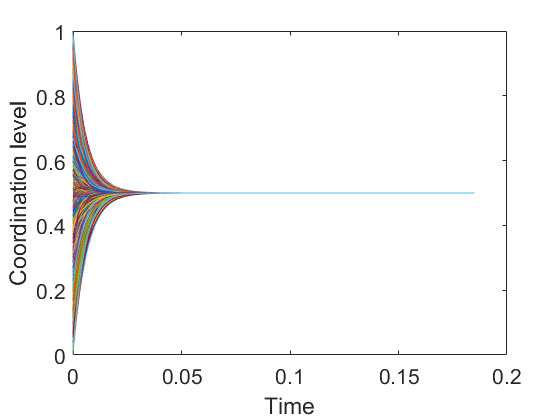}
\caption{Time evolution of coordination levels of 1000 agents.}
\label{fig:fig1}
\end{figure}
Figure \ref{fig:fig2} illustrates a logarithmic plot of the evolution of the normalized norm of the disagreement vector $\frac{\|\e_t\|_2}{{\|\e_0\|_2}}$, referred to as relative error, averaged over 1000 experiments (random topologies $\G'$ and initializations of $\x_0$); it is evident that the convergence is exponential, in full alliance with Theorem~\ref{thm:convergence}.
\begin{figure}
	\includegraphics[scale=0.47]{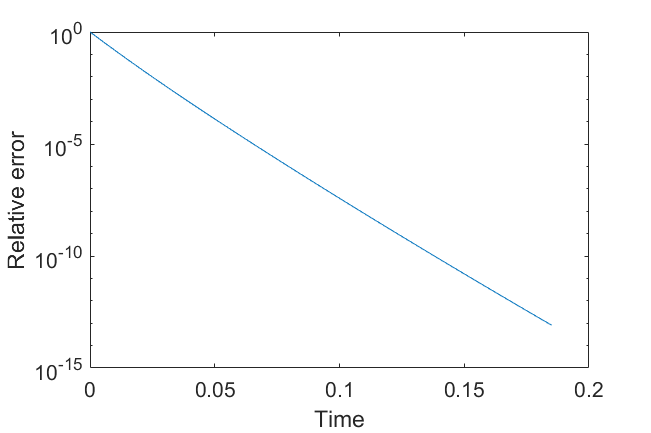}
	\caption{Normalized disagreement vector norm over time.}
	\label{fig:fig2}
\end{figure}

\section{Conclusions and future work}\label{sec:conclusions}
We have proposed and analyzed average consensus on evolutionary graphs. Linear consensus iterations are performed on a dynamic graph, where the topology is determined by an evolutionary game in which agents can randomly create new links or drop existing ones in a selfish manner based on their fitness function. We have established a.s. and m.s. average consensus with expected exponential rate regardless of the initial topology and agents' values. 

Our future work will focus on  devising distributed methods for multi-agent optimization over evolutionary graphs, as well as on extending the analysis to directed graphs and discrete-time iterations.

\bibliographystyle{IEEEtran}
\bibliography{references}

\end{document}

%% file: preamble.tex
\usepackage{amsmath,amssymb,amsfonts}
\usepackage{cases}
\usepackage{times,subfigure}
\usepackage{comment}
\usepackage{algpseudocode}
\usepackage{mathtools}
\usepackage{algorithm}
\usepackage{graphicx}
\usepackage[T1]{fontenc}
\usepackage[scanall]{psfrag}
\usepackage{epstopdf}
\usepackage{graphicx}
\usepackage{xcolor}

\newtheorem{Theorem}{Theorem}
\newtheorem{Lemma}{Lemma}
\newtheorem{Corollary}{Corollary}
\newtheorem{Proposition}{Proposition}
\newtheorem{Remark}{Remark}

\newtheorem{Assumption}{Assumption}

\def\sqw{\hfill\hbox{\lower.1ex\hbox{$\sqcup$}
		\kern-1.02em\lower.1ex\hbox{$\sqcap$}}\ }
\renewcommand{\vec}[1]{ \mathbf{#1}}

\DeclareMathOperator*{\rank}{rank}

\newcommand{\one}{\mathbf{1}}
\newcommand{\zero}{\mathbf{0}}
\newcommand{\R}{\mathbb{R}}

\newcommand{\hide}[1]{}

\newcommand\abs[1]{\left|#1\right|}

\newcommand{\vecx}[1]{\vec{#1}_{t}=\left \lbrace #1_{1,t}, \ldots, #1_{n,t} \right \rbrace}

\graphicspath{{./figures/}}

\definecolor{mygreen}{rgb}{0.0, 0.5, 0.0}

\newcommand{\N}{\mathcal{N}}

\newcommand{\x}{\mathbf{x}}

\newcommand{\z}{\mathbf{z}}
\newcommand{\e}{\mathbf{e}}

\newcommand{\W}{\mathsf{W}}

\newcommand{\G}{\mathcal{G}}
\newcommand{\V}{\mathcal{V}}
\newcommand{\Ed}{\mathcal{E}}

\newcommand{\F}{\mathcal{F}}

\newcommand{\E}{\mathbb{E}}
\renewcommand{\P}{\mathbb{P}}

\newcommand{\A}{\mathsf{A}}
\newcommand{\D}{\mathsf{D}}
\renewcommand{\L}{\mathsf{L}}

%% file: ECC18.bbl
\begin{thebibliography}{10}
\providecommand{\url}[1]{#1}
\csname url@samestyle\endcsname
\providecommand{\newblock}{\relax}
\providecommand{\bibinfo}[2]{#2}
\providecommand{\BIBentrySTDinterwordspacing}{\spaceskip=0pt\relax}
\providecommand{\BIBentryALTinterwordstretchfactor}{4}
\providecommand{\BIBentryALTinterwordspacing}{\spaceskip=\fontdimen2\font plus
\BIBentryALTinterwordstretchfactor\fontdimen3\font minus
  \fontdimen4\font\relax}
\providecommand{\BIBforeignlanguage}[2]{{%
\expandafter\ifx\csname l@#1\endcsname\relax
\typeout{** WARNING: IEEEtran.bst: No hyphenation pattern has been}%
\typeout{** loaded for the language `#1'. Using the pattern for}%
\typeout{** the default language instead.}%
\else
\language=\csname l@#1\endcsname
\fi
#2}}
\providecommand{\BIBdecl}{\relax}
\BIBdecl

\bibitem{tembine1}
H.~Tembine, E.~Altman, R.~El-Azouzi, and Y.~Hayel, ``Evolutionary games in
  wireless networks,'' \emph{IEEE Transactions on Systems, Man, and
  Cybernetics, Part B (Cybernetics)}, vol.~40, no.~3, pp. 634--646, 2010.

\bibitem{tembine2}
E.~Altman, R.~El-Azouzi, Y.~Hayel, and H.~Tembine, ``The evolution of transport
  protocols: An evolutionary game perspective,'' \emph{Computer Networks},
  vol.~53, no.~10, pp. 1751--1759, 2009.

\bibitem{sun2015}
C.~Sun and H.~Duan, ``Markov decision evolutionary game theoretic learning for
  cooperative sensing of unmanned aerial vehicles,'' \emph{Sci China Tech Sci},
  vol.~58, pp. 1392--1400, 2015.

\bibitem{tembine3}
H.~Tembine and A.~P. Azad, ``Dynamic routing games: An evolutionary game
  theoretic approach,'' in \emph{50th IEEE Conference on Decision and Control
  and European Control Conference (CDC-ECC)}, 2011, pp. 4516--4521.

\bibitem{linear_cons}
J.~Tsitsiklis and M.~Athans, ``Convergence and asymptotic agreement in
  distributed decision problems,'' \emph{IEEE Transactions on Automatic
  Control}, vol.~29, no.~1, pp. 42--50, 1984.

\bibitem{xiao2004fast}
L.~Xiao and S.~Boyd, ``Fast linear iterations for distributed averaging,''
  \emph{Systems \& Control Letters}, vol.~53, no.~1, pp. 65--78, 2004.

\bibitem{boyd2006randomized}
S.~Boyd, A.~Ghosh, B.~Prabhakar, and D.~Shah, ``Randomized gossip algorithms,''
  \emph{IEEE/ACM Transactions on Networking}, vol.~14, pp. 2508--2530, 2006.

\bibitem{olfati}
R.~Olfati, Saber and R.~Murray, ``Agreement problems in networks with directed
  graphs and switching topology,'' in \emph{Proceedings of the 42nd IEEE
  Conference on Decision and Control (CDC)}, 2003, pp. 4126--4132.

\bibitem{nedic2009distributed}
A.~Nedi\c{c} and A.~Ozdaglar, ``Distributed subgradient methods for multi-agent
  optimization,'' \emph{IEEE Transactions on Automatic Control}, vol.~54,
  no.~1, pp. 48--61, 2009.

\bibitem{nedich2016achieving}
A.~Nedi\c{c}, A.~Olshevsky, and W.~Shi, ``Achieving geometric convergence for
  distributed optimization over time-varying graphs,'' \emph{arXiv preprint
  arXiv:1607.03218}, 2016.

\bibitem{gossip}
A.~G. Dimakis, S.~Kar, J.~M. Moura, M.~G. Rabbat, and A.~Scaglione, ``Gossip
  algorithms for distributed signal processing,'' \emph{Proceedings of the
  IEEE}, vol.~98, no.~11, pp. 1847--1864, 2010.

\bibitem{REK}
A.~Zouzias and N.~Freris, ``{Randomized Extended Kaczmarz} for solving least
  squares,'' \emph{SIAM Journal on Matrix Analysis and Applications}, vol.~34,
  no.~2, pp. 773--793, 2013.

\bibitem{RK}
N.~Freris and A.~Zouzias, ``Fast distributed smoothing of relative
  measurements,'' in \emph{Proceedings of the 51st IEEE Conference on Decision
  and Control (CDC)}, 2012, pp. 1411--1416.

\bibitem{RK2}
A.~Zouzias and N.~Freris, ``Randomized gossip algorithms for solving
  {L}aplacian systems,'' in \emph{Proceedings of the 14th IEEE European Control
  Conference (ECC)}, 2015, pp. 1920--1925.

\bibitem{sn1}
T.~Y. Berger-Wolf and J.~Saia, ``A framework for analysis of dynamic social
  networks,'' in \emph{Proceedings of the 12th ACM SIGKDD International
  Conference on Knowledge Discovery and Data Mining}, 2006, pp. 523--528.

\bibitem{graph_theory}
D.~West, \emph{Introduction to graph theory}.\hskip 1em plus 0.5em minus
  0.4em\relax {Prentice Hall Upper Saddle River}, 2001, vol.~2.

\bibitem{egraph1}
E.~Lieberman, C.~Hauert, and M.~A. Nowak, ``Evolutionary dynamics on graphs,''
  \emph{Nature}, vol. 433, no. 7023, pp. 312--316, 2005.

\bibitem{egraph2}
G.~Szab{\'o} and G.~Fath, ``Evolutionary games on graphs,'' \emph{Physics
  reports}, vol. 446, no.~4, pp. 97--216, 2007.

\bibitem{mccabe2008controlled}
R.~McCabe, N.~Freris, and P.~R. Kumar, ``Controlled random access {MAC} for
  network utility maximization in wireless networks,'' in \emph{Proceedings of
  the 47th IEEE Conference on Decision and Control (CDC)}, 2008, pp.
  2350--2355.

\bibitem{freris2010fundamentals}
N.~Freris, H.~Kowshik, and P.~R. Kumar, ``Fundamentals of large sensor
  networks: Connectivity, capacity, clocks, and computation,''
  \emph{Proceedings of the IEEE}, vol.~98, no.~11, pp. 1828--1846, 2010.

\bibitem{caipd}
B.~Ranjbar-Sahraei, H.~Bou~Ammar, D.~Bloembergen, K.~Tuyls, and G.~Weiss,
  ``Evolution of cooperation in arbitrary complex networks,'' in
  \emph{Proceedings of the 2014 International Conference on Autonomous Agents
  and Multiagent Systems}, 2014, pp. 677--684.

\bibitem{ecc}
M.~Smyrnakis, D.~Bauso, P.~A. Trodden, and S.~M. Veres, ``Learning of
  cooperative behaviour in robot populations,'' in \emph{Proceedings of the
  15th IEEE European Control Conference (ECC)}, 2016, pp. 184--189.

\bibitem{stroock}
D.~Stroock, \emph{Probability theory}.\hskip 1em plus 0.5em minus 0.4em\relax
  {Cambridge University Press}, 1994.

\bibitem{royden}
H.~L. Royden, \emph{Real analysis}, 3rd~ed.\hskip 1em plus 0.5em minus
  0.4em\relax Macmillan New York, 1989.

\bibitem{watts1999small}
D.~Watts, \emph{Small worlds: the dynamics of networks between order and
  randomness}.\hskip 1em plus 0.5em minus 0.4em\relax Princeton University
  Press, 1999.

\end{thebibliography}
